\documentclass[journal]{journal}
\ifCLASSINFOpdf
\else
\fi
\usepackage{authblk}
\usepackage{hyperref}
\hypersetup{urlcolor=blue, citecolor=red}
\usepackage{amsmath,amsthm,amssymb,xcolor,bbm,mathrsfs}
\usepackage[english]{babel}
\usepackage{fancyhdr}
\newcommand{\norm}[1]{\|#1\|}
\newcommand{\n}{\hspace*{-6pt}}

\theoremstyle{definition}
\newtheorem{theorem}{Theorem}
\newtheorem{definition}{Definition}
\newtheorem{lemma}{Lemma}
\newtheorem{corollary}{Corollary}

\usepackage{graphicx, subfigure, float, caption}

\begin{document}
%
\title{Mixed Leader-Follower Dynamics}
%
%
%

\author{Hsin-Lun Li
\thanks{H Li is with the Department of Applied Mathematics, National Sun Yat-sen University, Kaohsiung 80424, Taiwan, e-mail: hsinlunl@asu.edu.}
\thanks{}
\thanks{}}

%
%

\markboth{Hsin-Lun Li}%
{Mixed Leader-Follower Dynamics}
%



\maketitle
\thispagestyle{empty}

\begin{abstract}
 The original Leader-Follower (LF) model partitions all agents whose opinion is a number in $[-1,1]$ to a follower group, a leader group with a positive target opinion in $[0,1]$ and a leader group with a negative target opinion in $[-1,0]$. A leader group agent has a constant degree to its target and mixes it with the average opinion of its group neighbors at each update. A follower has a constant degree to the average opinion of the opinion neighbors of each leader group and mixes it with the average opinion of its group neighbors at each update. In this paper, we consider a variant of the LF model, namely the mixed model, in which the degrees can vary over time, the opinions can be high dimensional, and the number of leader groups can be more than two. We investigate circumstances under which all agents achieve a consensus. In particular, a few leaders can dominate the whole population.
\end{abstract}


%
\IEEEpeerreviewmaketitle

\section{Introduction}
%
%
%
%
\IEEEPARstart{T}{he} Hegselmann-Krause (HK) model and the Deffuant model are two of the most popular mathematical models in opinion dynamics. The similarities between the two are that both consist of a finite number of agents and set a constant confidence threshold, therefore belonging to the bounded confidence model. \cite{proskurnikov2017tutorial, hegselmann2002opinion, deffuant2000mixing, bhattacharyya2013convergence, lorenz2005stabilization, lorenz2007continuous, vasca2021practical, bernardo2021heterogeneous} and \cite{chen2020convergence} are some papers relating to the bounded confidence model. In both models, all agents have their opinion in $\mathbf{R}^d$ and two agents are opinion neighbors if their opinion distance does not exceed the confidence threshold. The original opinion space of both models is the closed interval between $0$ and $1$, namely $[0,1]$. The difference between the two is their interaction mechanisms. The former belongs to group interaction, whereas the latter belongs to pair interaction. 

The HK model is averaging dynamics, divided into two types: the synchronous HK model and the asynchronous HK model. The updated opinion of an agent is the average opinion of its opinion neighbors. All agents update their opinion simultaneously for the synchronous HK model and only one agent uniformly selected at random updates its opinion for the asynchronous model. \cite{etesami2015game} is a paper about the synchronous HK model and the asynchronous HK model. Some variants of the HK model were proposed, such as \cite{fu2015opinion, parasnis2018hegselmann} and \cite{fortunato2005consensus}, which belong to either the model with limited interactions or that with self-belief. The authors in \cite{lanchier2022consensus} raised a nontrivial lower bound for the probability of consensus of the HK model with continuous time. The author in \cite{mHK} proposed a variant of the HK model in which all agents can decide their degree of stubbornness and mix their opinion with the average opinion of their opinion neighbors, called the mixed HK model. The degree of stubbornness is a number between $0$ and $1$. The more stubborn an agent, the closer to $1$ its degree of stubbornness, and vice versa. The mixed HK model is the synchronous HK model when all agents are absolutely open-minded, and the asynchronous HK model, when only one agent uniformly selected at random is absolutely open-minded and the rest are absolutely stubborn. The author showed circumstances under which a consensus can be achieved or asymptotic stability holds. The techniques include Perron-Frobenius for Laplacians in \cite{biyikoglu2007laplacian}, Courant-Fischer Formula in \cite{horn2012matrix} and Cheeger's Inequality in \cite{beineke2004topics}. \cite{mHK2} is the sequel to \cite{mHK}. The mixed HK model is studied deterministically in \cite{mHK} and nondeterministically in \cite{mHK2}. 

Apart from the HK model, all agents in the Deffuant model interact in pairs and there is a social relationship, such as a friendship, among some agents. Two agents are social neighbors if they have a social relationship. A pair of social neighbors are selected at each time step and interact if and only if their opinion distance does not exceed the confidence threshold. The author in \cite{MR3069370} proposed the first proof of the main conjecture about the Deffuant model. An alternative proof to that conjecture is in \cite{MR2915577}. \cite{MR3164772}, \cite{MR3265084}, \cite{MR3694315} and \cite{MR4058367} are works related to the Deffuant model. The authors in \cite{lanchier2020probability} raised a nontrivial lower bound, which does not depend on the size of all agents, for the probability of consensus of the Deffuant model. Some properties of the HK model and the Deffuant model are in common. For instance, all agents are opinion-connected to each other thereafter if they are opinion-connected to each other at some time step. The author in \cite{mHK2} blended a social relationship in the mixed HK model and argued that the mixed HK model includes group interaction and pair interaction, therefore containing the HK model and the Deffuant model.

The  Leader-Follower (LF) dynamics originated from the Hegselmann-Krause (HK) dynamics. The authors in \cite{zhao2018understanding} proposed the LF model that partitions agents whose opinion is in $[-1,1]$ into a follower group, a leader group with a positive target opinion in $[0,1]$ and a leader group with a negative target opinion in $[-1,0]$. Namely, all agents consist of three types of agents: followers, positive target leaders and negative target leaders. In the original HK model, all agents update their opinion in $[0,1]$ by taking the average opinion of their opinion neighbors. In the LF model, all positive target leaders have a constant degree toward their positive target and mix the average opinion of their positive target neighbors with their positive target, all negative target leaders have a constant degree toward their negative target and mix the average opinion of their negative target neighbors with their negative target, and all followers have their constant degree toward the average opinion of their positive target neighbors and that toward the average opinion of their negative target neighbors and mix the average opinion of their follower neighbors with the average opinion of their positive target neighbors and the average opinion of their negative target neighbors. In a nutshell, the leader groups consider their group neighbors and their target, whereas the follower group follows others, therefore considering follower group neighbors, positive target group neighbors and negative target group neighbors. Interpreting mathematically, define $[n]=\{1,2,\ldots,n\}$. Say $N$ agents including $N_1$ followers, $N_2$ positive target agents and $N_3$ negative target agents, set as $[N_1]$, $[N_1+N_2]-[N_1]$ and $[N]-[N_1+N_2]$.
The mechanism is as follows:
$$\begin{array}{l}
\displaystyle x_i(t+1)=\frac{1-\alpha_i-\beta_i}{|N_i^F(t)|}\sum_{j\in N_i^F(t)}x_j(t)\vspace{2pt}\\
\displaystyle \hspace{55pt}+ \frac{\alpha_i}{|N_i^P(t)|}\sum_{j\in N_i^P(t)}x_j(t) \vspace{2pt}\\
\displaystyle \hspace{55pt}+ \frac{\beta_i}{|N_i^N(t)|}\sum_{j\in N_i^N(t)}x_j(t),\quad i=1,\ldots,N_1,\vspace{2pt}\\
 \displaystyle   x_i(t+1)=\frac{(1-w_i)}{|N_i^P(t)|}\sum_{j\in N_i^P(t)}x_j(t)+ w_i d,\\ \vspace{2pt} \hspace{2cm} i=N_1+1,\ldots,N_1+N_2, \vspace{2pt}\\
 \displaystyle x_i(t+1)=\frac{1-z_i}{|N_i^N(t)|}\sum_{j\in N_i^N(t)}(t)x_j(t)+ z_i g,\\ \vspace{2pt}\hspace{2cm} i=N_1+N_2+1,\ldots,N,

\end{array}$$
where 
$$\begin{array}{rcl}
\displaystyle x_i(t) &\n=\n& \hbox{opinion of agent $i$ at time $t$},\vspace{2pt}\\
  \displaystyle   d &\n\in\n& [0,1]\ \hbox{is the positive target opinion},\vspace{2pt} \\
 \displaystyle    g &\in& [-1,0]\ \hbox{is the negative target opinion}, \vspace{2pt}\\
\displaystyle \epsilon_i &\n=\n& \hbox{confidence threshold of agent }i,\vspace{2pt}\\
\displaystyle N_i^F(t)&\n=\n& \{j\in [N_1]:\norm{x_i(t)-x_j(t)}\leq \epsilon_i\}, \vspace{2pt}\\
  N_i^P(t)&\n=\n& \{j\in [N_1+N_2]-[N_1]:\vspace{2pt}\\
  &&\displaystyle\hspace{2cm}\norm{x_i(t)-x_j(t)}\leq \epsilon_i\}, \vspace{2pt}\\
\displaystyle   N_i^N(t)&\n=\n& \{j\in [N]-[N_1+N_2]:\vspace{2pt}\\
&&\displaystyle\hspace{2cm}\norm{x_i(t)-x_j(t)}\leq \epsilon_i\}, \vspace{2pt}\\
\displaystyle   \alpha_i &\n = \n& \hbox{degree to the average opinion of agent $i$'s} \\
&&\hbox{positive target neighbors}, \vspace{2pt}\\
\displaystyle   \beta_i &\n = \n& \hbox{degree to the average opinion of agent $i$'s}\\
&&\hbox{negative target neighbors}, \vspace{2pt}\\
   w_i &\n = \n& \hbox{degree to the positive target of agent }i, \vspace{2pt}\\
\displaystyle   z_i &\n = \n& \hbox{degree to the negative target of agent }i, \vspace{2pt}\\
   \alpha_i,\ &\beta_i,&\ w_i,\quad z_i \in [0,1].
\end{array}$$ 
The authors in \cite{zha2020opinion} pointed out that it can be an application in e-commerce. In this paper, we consider a variant of the LF model, namely the mixed LF model. The differences are:

\begin{itemize}
    \item opinions can be high dimensional,
    \item the number of leader groups can be more than two,
    \item the degree of a leader group agent to its group target and the degree of a follower group agent to the average opinion of each of its leader group neighbors can vary over time,
    \item only one confidence threshold, $\epsilon$.
\end{itemize}
 In the leader group, each agent can choose its degree to the target opinion and mix its opinion with the average opinion of its group neighbors at each update. In the follower group, each agent can choose its degree to the average opinion of its neighbors in each leader group and mix its opinion with the average opinion of its group neighbors at each update. In a nutshell, leader group members only can interact in its group, whereas follower group members can interact out of its group. Considering a finite set of agents partitioned to a follower group, $F$, and $m$ leader groups, $L_1, \ldots, L_m$, the mixed model is as follows:
 $$\begin{array}{l}
  \displaystyle    x_i(t+1)=\frac{\alpha_i^k(t)}{|N_i^{L_k}(t)|}\sum_{j\in  N_i^{L_k}(t)}x_j(t)+(1-\alpha_i^k(t))g_k,\\ \vspace{2pt}\hspace{2cm} i\in L_k, \vspace{2pt}  \\
   \displaystyle   x_i(t+1)=\frac{\big(1-\sum_{k=1}^m\beta_i^k(t)\big)}{|N_i^F(t)|}\sum_{k\in N_i^F(t)}x_k\\ \vspace{2pt}\hspace{2cm}\displaystyle+\sum_{k=1}^m\frac{\beta_i^k(t)}{|N_i^{L_k}(t)|}\sum_{k\in N_i^{L_k}(t)}x_k, \quad  i\in F, 
 \end{array}$$
where
$$\begin{array}{rcl}
   \displaystyle  \alpha_i^k(t)&\n\in\n&[0,1]\hbox{ is the degree to the average opinion of} \\
   &&\hbox{agent $i$'s group neighbors}, \vspace{2pt}\\
   \displaystyle g_k&\n \in\n &\mathbf{R^d}\hbox{ is the target opinion of leader group }k, \vspace{2pt}\\
N_i^F(t) &\n=\n& \{j\in F:\norm{x_j(t)-x_i(t)}\leq\epsilon\}, \vspace{2pt}  \\
     N_i^{L_k}(t) &\n=\n& \{j\in L_k:\norm{x_j(t)-x_i(t)}\leq\epsilon\}, \vspace{2pt}\\
     \beta_i^k(t) &\n=\n& \hbox{degree to the average opinion of agent $i$'s}\\
     &&\hbox{neighbors in leader group $k$,}\vspace{2pt}\\
     \beta_i^k(t) &\n\in\n& [0,1]\ \hbox{and equals}\ 0\ \hbox{if}\ N_i^{L_k}(t)=\emptyset .
\end{array}$$ 
Observe that the system of leader group $k$ becomes a synchronous HK model if $\alpha_i^k(t)=1$ for all $i\in L_k$, and the follower group system becomes a synchronous HK model if $\beta_i^k(t)=0$ for all $i\in F$ and $k\in [m]$.
The differences between HK dynamics and LF dynamics are:
\begin{itemize}
    \item HK dynamics have only one group and one interaction rule, whereas LF dynamics has two types of groups, the leader group and the follower group, in which each has its interaction rule, 
    \item the update of an agent in HK dynamics only depends on its neighbors, whereas the update of an agent in LF dynamics may also depend on a target.
\end{itemize}

 
\section{Main results}
The next theorem shows circumstances under which all agents in a leader group achieve their target. It turns out that the degree to their target, $1-\alpha_i(t)$, can be very small to achieve their target for all agents in $L$.

\begin{theorem}\label{tar}
Assume that $\limsup_{t\to\infty}\max_{i\in L}\alpha_i(t)<1$. Then, $$\lim_{t\to\infty}\max_{i\in L}\norm{x_i(t)-g}=0.$$
In particular, $\lim_{t\to\infty}\norm{x_i(t)-g}=0$ if $\lim_{t\to\infty}\alpha_i(t)=0$ for some $i\in L$.
\end{theorem}
Next, we investigate circumstances under which all agents achieve a consensus. It turns out that a consensus can be achieved as long as their opinion lies in $B(g,\delta)$ for some $\delta<\epsilon$ at some time step and the minimum degree of all leader group agents to their goal and the minimum degree of all followers to the average opinion of their leader neighbors have a lower bound after some time step.
\begin{theorem}\label{cs}
Assume that $\big\{x_i(t)\big\}_{i\in L\cup  F}\subset B(g,\epsilon) $ for some $t\geq 0$ and that $$\sup_{s\geq t}\big\{\max_{i\in F}(1-\beta_i(s)),\ \max_{i\in L}\alpha_i(s)\big\}<1.$$ Then, $$\lim_{t\to \infty}\max_{i\in  L\cup F} \norm{x_i(t)-g}=0.$$
\end{theorem}
Considering an LF system consisting of a follower group and $m$ leader group, we have the following corollaries. Given that the opinions of all agents and all targets fall in some open ball centered at some target of radius $\epsilon$ and that the maximum tendency of all leaders toward their leader group opinion neighbors and all followers toward their follower group opinion neighbors is less than some constant less than 1 after some time step. Then, the distance between each agent's opinion and its convex combination of all targets approaches 0 over time.

\begin{corollary}
Assume that $$ \big\{x_i(t),\ g_k\big\}_{i\in (\bigcup_{k=1}^m L_k)\cup F,\ k\in [m]}\subset B(g_j,\epsilon)$$ for some $j\in[m]$ and $t\geq 0$ and that $$\sup_{s\geq t}\big\{\max_{i\in F}(1-\sum_{k=1}^m\beta_i^k(s)),\ \max_{k\in [m]} \max_{i\in L_k}\alpha_i^k(s)\big\}<1.$$ Then, $$\lim_{t\to \infty}\max_{i\in  (\bigcup_{k=1}^m L_k)\cup F} \norm{x_i(t)- \frac{\sum_{k=1}^m\beta_i^k(t)  g_k}{\sum_{j=1}^m\beta_i^j(t)} }=0.$$
In particular for all $ i\in  (\bigcup_{k=1}^m L_k)\cup F$, $$ \lim_{t\to \infty}x_i(t)=  \frac{\sum_{k=1}^m\beta_i^k g_k}{\sum_{j=1}^m\beta_i^j} $$ if $\beta_i^k(t)\to \beta_i^k\hbox{ as }t\to \infty\hbox{ for all }k\in[m].$
\end{corollary}
Given that all leader groups are independent systems in which all followers belong to one of them and that the maximum tendency of all leaders toward their leader group opinion neighbors and all followers toward their follower group opinion neighbors is less than some constant less than 1 after some time step, then all systems achieve their target over time.

\begin{corollary}
Assume that $\min_{i,j\in [m]}\norm{g_i-g_j}>3\epsilon$, $(x_i(t))_{i\in L_k}\subset B(g_k,\epsilon)$\ \hbox{for some}\ $t\in \mathbf{N}$\ \hbox{and for all}\ $k\in [m]$ and $x_j(t)\in B(g_k,\epsilon)$ for all $j\in F$ and for some $k\in [m]$, and that $$\sup_{s\geq t}\big\{\max_{i\in F}(1-\sum_{k=1}^m\beta_i^k(s)),\ \max_{k\in [m]} \max_{i\in L_k}\alpha_i^k(s)\big\}<1.$$ Then, all systems eventually achieve their target.
\end{corollary}

\section{The model}
We first investigate the behavior of a leader group. Let $L$ be a leader group and $g$ be its target. It turns out that the maximum distance between its opinions and target, $\max_{i\in L}\norm{x_i(t)-g}$, is nonincreasing over time. 
\begin{lemma}\label{rec}
We have $$\norm{x_i(t+1)-g}\leq \alpha_i(t)\max_{j\in N_i(t)}\norm{x_j(t)-g}\ \hbox{for all }i\in L.$$ 
In particular,
$$\max_{i\in L}\norm{x_i(t+1)-g}\leq\max_{i\in L}\alpha_i(t)\max_{i\in L}\norm{x_i(t)-g}.$$
\end{lemma}

\begin{proof}
Observe that 
\begin{align*}
    \norm{x_i(t+1)-g}&=\frac{\alpha_i(t)}{|N_i(t)|}\norm{\sum_{j\in N_i(t)}[x_j(t)-g]}\\
    &\leq\alpha_i(t)\max_{j\in N_i(t)}\norm{x_j(t)-g}.
\end{align*}
Therefore,
$$\max_{i\in L}\norm{x_i(t+1)-g}\leq\max_{i\in L}\alpha_i(t)\max_{i\in L}\norm{x_i(t)-g}.$$
\end{proof}

\begin{proof}[\bf Proof of Theorem \ref{tar} ]
Since $\limsup_{t\to\infty}\max_{i\in L}\alpha_i(t)<1$, there is $(t_k)_{k\geq0}\subset\mathbf{N}$ strictly increasing such that $\max_{i\in L}\alpha_i(t_k)\leq\delta<1$ for some $\delta$. For all $t\geq1$, $t_s< t\leq t_{s+1}$ for some $s\in\mathbf{N}$. Via Lemma \ref{rec},
\begin{align*}
   & \max_{i\in L}\norm{x_i(t)-g}\\
    &\hspace{6pt}\leq\max_{i\in L}\alpha_i(t-1)\ldots\max_{i\in L}\alpha_i(t_0)\max_{i\in L}\norm{x_i(t_0)-g}\\
    &\hspace{6pt}\leq \delta^{s+1}\max_{i\in L}\norm{x_i(t_0)-g}.
\end{align*}
As $t\to\infty$, $s\to\infty$. Therefore, $$\limsup_{t\to\infty}\max_{i\in L}\norm{x_i(t)-g}\leq 0.$$ Hence, $$\lim_{t\to\infty}\max_{i\in L}\norm{x_i(t)-g}=0.$$
Also, due to Lemma \ref{rec} and $\max_{j\in L}\norm{x_j(t)-g}$ bounded over time,
$$\limsup_{t\to\infty}\norm{x_i(t+1)-g}\leq 0\ \hbox{therefore }\lim_{t\to\infty}\norm{x_i(t)-g}=0.$$

\end{proof}
In other words, all leader group agents achieve their target as long as the minimum degree to their target has a lower bound larger than 0 infinitely many times. In particular, a leader group agent achieves the group target if its degree to the target approaches 1 over time. Now, considering an LF system consisting of a follower group F and a leader group L, let $B(c,r)$ be an open ball centered at $c$ of radius $r$. 
\begin{definition}
The convex hull generated by $v_1,\ldots,v_n$, $C(\{v_1,\ldots,v_n\})$, is the smallest convex set containing $v_1,\ldots,v_n$, i.e., $$C(\{v_1,\ldots,v_n\})=\{\sum_{i=1}^n a_i v_i:(a_i)_{i=1}^n\ \hbox{is stochastic}\}.$$

\end{definition}

Observe that $x_i(t+1)\in C(\{x_k(t)\}_{k\in L\cup F}\cup \{g\})$ for all $i\in L\cup F$; therefore we have the following lemma. 

\begin{lemma}\label{cir}
For all $\delta\geq 0$, if $$\big\{x_k(t)\big\}_{k\in L\cup F}\subset B(g,\delta),\ \hbox{then}\ \big\{x_k(t+1)\big\}_{k\in L\cup F}\subset B(g,\delta).$$
\end{lemma}

\begin{proof}[\bf Proof of Theorem \ref{cs}]
Via Theorem \ref{tar},  $$\max_{k\in L}\norm{x_k(s)-g}<\delta\ \hbox{for all $\delta>0$,  for some $p\geq t$}$$  and for all $s\geq p.$
For all $s\geq p$, $\delta>0$, $i\in F$ and $j\in L$, via Lemma \ref{cir}, we have $$\norm{x_i(s)-x_j(s)}\leq \norm{x_i(s)-g}+\norm{g-x_j(s)}< \epsilon+\delta $$ 
therefore $\norm{x_i(s)-x_j(s)}\leq \epsilon$ and $N_i^L(s)=L$.
Let $\alpha_t=\max_{k\in L}\alpha_k(t)$, $\beta_t=\max_{k\in F}(1-\beta_k(t))$, $\gamma=\sup_{s\geq t}\{\beta_s,\ \alpha_s\}$,  $A_t=\max_{k\in F}\norm{x_k(t)-g}$ and $C_t=\max_{k\in L}\norm{x_k(t)-g}$.
By Lemma \ref{rec} and the triangle inequality, for all $i\in F$ and $t> p$,
\begin{align*}
  & A_{t+1}\leq\  \beta_t A_t+ C_t\\
  &\hspace{0.4cm} \leq\  \beta_t\beta_{t-1}\ldots\beta_p A_{p}+\beta_t\ldots\beta_{p+1}C_p+\ldots+\beta_t C_{t-1} + C_{t}\\
 &\hspace{0.4cm}  \leq\  \gamma^{t-p+1}A_p+(t-p+1)\gamma^{t-p}C_p
\end{align*}
therefore
$$\limsup_{t\to\infty}A_{t+1}\leq 0.$$
Via Theorem \ref{tar}, we are done.
\end{proof}

\section{Conclusion}
All agents of a leader group achieve their consensus as long as there are infinitely many times that the most stubborn are not too stubborn. For a leader group and a follower group, all agents eventually achieve the leader group's target if all agents' opinion is within $\epsilon$ distance from the leader group's target and all agents are not too stubborn. In particular, a few leaders can dominate the whole population.   

\section{Simulations}
For simulations of Theorem~\ref{tar}, considering a leader group of size 100, all opinions uniform on $[-1,0]$, confidence threshold equal to 0.05, target equal to 1 and all agents' tendencies to their leader group opinion neighbors having 1/3 of chance being 0.99 and 1 otherwise. Then, all leader group agents achieve their target eventually as Figure~\ref{fig:leader_group} shows. For simulations of Theorem~\ref{cs}, considering a leader group of size 2 with opinions 0.99 and -0.99, a follower group of size 100 with opinions uniform on $[0.5, 1]$, confidence threshold equal to 1, the target of the leader group equal to 0 and the tendencies of all leader group agents toward their leader group opinion neighbors and that of all followers toward their follower group opinion neighbors are 0.99 at all times. Then, all agents achieve the target eventually as Figure~\ref{fig:leader_follower_group} shows.

\begin{figure}[H]
\centering

\includegraphics[width=2.5in]{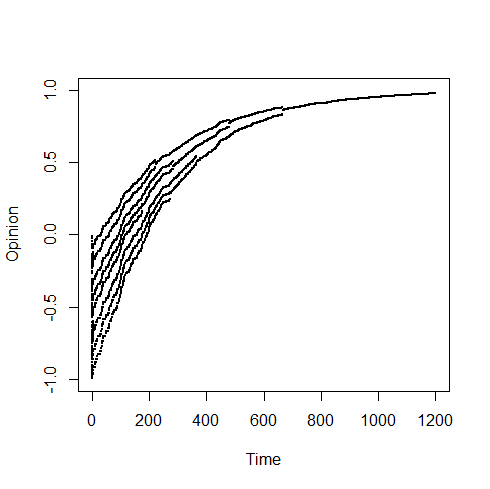}
\caption{Interaction of a leader group}
\label{fig:leader_group}
\end{figure}

\begin{figure}[H]
\centering

\includegraphics[width=2.5in]{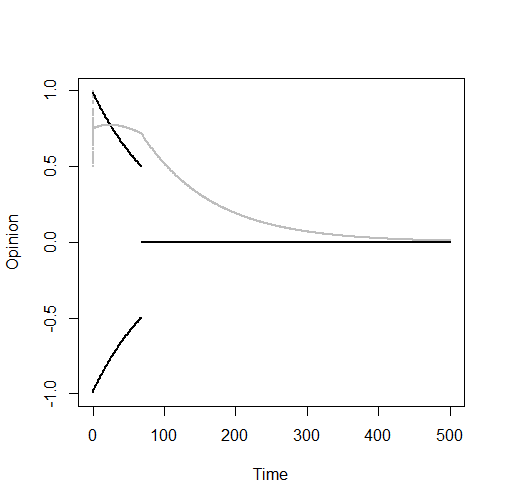}
\caption{Interaction of a leader group and a follower group}
\label{fig:leader_follower_group}
\end{figure}


%



\section*{Acknowledgment}
The author is funded by the National Science and Technology Council.


\ifCLASSOPTIONcaptionsoff
  \newpage
\fi

\end{document}